\documentclass[12pt]{article}
\usepackage{mathrsfs}
\usepackage{mathrsfs}
\usepackage{amsmath, amsthm, amsfonts, amssymb}
\usepackage{mathrsfs}
\usepackage{bbm}
\usepackage{amssymb}
\usepackage{txfonts}

\newlength{\defbaselineskip}
\newcommand{\setlinespacing}[1]%
           {\setlength{\baselineskip}{#1 \defbaselineskip}}

\newcounter{marnote}

\title{Derivative estimates of solutions of elliptic systems in narrow regions}
\author{HaiGang Li\footnote{School of
Mathematical Sciences, Beijing Normal University, Laboratory of
Mathematics and Complex Systems, Ministry of Education, Beijing
100875, China. Email: hgli@bnu.edu.cn},~ YanYan
Li\footnote{Department of Mathematics, Rutgers University, 110
Frelinghuysen Rd, Piscataway, NJ 08854, USA. Email:
yyli@math.rutgers.edu},~ Ellen ShiTing Bao\footnote{Department of Mathematics, University of Pittsburgh,
Pittsburgh, PA 15260, USA.},~ and Biao Yin\footnote{Quantitative Advisory Services, Ernst and Young LLP, 5 Times Square, New York, NY 10036, USA.}}

\date{}

\begin{document}
\newtheorem{Def}{Definition}[section]
\newtheorem{thm}{Theorem}[section]
\newtheorem{lem}{Lemma}[section]
\newtheorem{rem}{Remark}[section]
\newtheorem{prop}{Proposition}[section]

\newtheorem{theoremAlph}{Theorem}\renewcommand\thetheoremAlph{\Alph{theoremAlph}}

\newtheorem{cor}{Corollary}[section]
\def\av{{\int \hspace{-2.25ex}-} }
\maketitle

\section{Introduction}\label{introduction}

In this paper we establish local derivative estimates
for solutions to a class of elliptic systems arising from studies of fiber-reinforced
composite materials. From the structure of the composite, there are a relatively
large number of fibers which are touching or nearly touching. The maximal strains can be strongly influenced by the distances between
the fibers.

Stimulated by some works on damage analysis of fiber composites (\cite{BASL}), there have been a number of papers, starting from \cite{BV}, \cite{LN} and \cite{LV}, on gradient estimates for solutions of elliptic equations and systems with piecewise smooth coefficients which are relevant in such studies. See, e.g. \cite{ACK,ADKL,AKL,AKLLL,AKLLZ}, \cite{BLY1,BLY2}, \cite{BT}, \cite{KLY}, \cite{LY}, \cite{Y1,Y2}. Earlier studies on such and closely related issues can be found in \cite{BC}, \cite{K1,K2}.

In a recent paper \cite{ADKL}, some gradient estimates were obtained concerning the conductivity problem where the conductivity is allowed to be $\infty$ (perfect conductor).

\begin{theoremAlph}\label{thmak}(\cite{ADKL})
Let $B_{1}$ and $B_{2}$ be two balls in $\mathbb{R}^{3}$ with radius $R$ and centered
at $(0,0,\pm{R}\pm\frac{\epsilon}{2})$, respectively. Let $H$ be a
harmonic function in $\mathbb{R}^{3}$ such that $H(0)=0$. Define $u$
to be the solutions of
\begin{equation}\label{equ_ak}
\begin{cases}
\Delta{u}=0&\mbox{in}~\mathbb{R}^{3}\setminus\overline{B_{1}\cup{B}_{2}},\\
u=0&\mbox{on}~\partial{B}_{1}\cup\partial{B}_{2},\\
u(x)-H(x)=O(|x|^{-1})&\mbox{as}~|x|\rightarrow+\infty.
\end{cases}
\end{equation}
Then there exists a constant $C$ independent of $\epsilon$ such that
\begin{equation}\label{adkl_u=0}
\|\nabla(u-H)\|_{L^{\infty}(\mathbb{R}^{3}\setminus\overline{B_{1}\cup{B}_{2}})}\leq\,C.
\end{equation}
\end{theoremAlph}

Contrary to scalar equations, less is
known about derivative estimates of solutions of systems. In this paper we extend Theorem \ref{thmak} to general elliptic systems, including linear systems of elasticity, in all dimensions. Moreover, we allow the two balls in Theorem \ref{thmak} to be replaced by any two smooth domains, and we establish a stronger local version.

We use
$B_{r}(0')=\{x'\in\mathbb{R}^{n-1}~|~|x'|<r\}$ to denote a ball in $\mathbb{R}^{n-1}$ centered at the
origin $0'$ of radius $r$. Let $h_{1}$ and $h_{2}$ be smooth functions in $B_{1}(0')$
satisfying
$$h_{1}(0')=h_{2}(0')=0,\quad\,\nabla{h}_{1}(0')=\nabla{h}_{2}(0')=0,$$
and
$$-\frac{\epsilon}{2}+h_{2}(x')<\frac{\epsilon}{2}+h_{1}(x'),\quad\mbox{for}~~|x'|<1.$$
For $0<r\leq\,1$, we define
\begin{equation*}
 \Omega_r:=\left\{x\in \mathbb{R}^n~\big|~-\frac{\epsilon}{2}+h_{2}(x')<x_{n}<\frac{\epsilon}{2}+h_{1}(x'),~x'\in{B}_{r}(0')\right\}.
\end{equation*}
Its lower and upper boundaries are, respectively,
$$ \Gamma_{r}^{-}=\left\{x\in\mathbb{R}^n~\big|~x_n=-\frac{\epsilon}{2}+h_{2}(x'),~|x'|\leq{r}\right\},\quad
\Gamma_{r}^{+}=\left\{x\in \mathbb{R}^n~\big|~x_n=\frac{\epsilon}{2}+h_{1}(x'),~|x'|\leq{r}\right\}.$$

Let $u=(u^{1},\cdots,u^{N})$ be a vector-valued
function. We consider
the following boundary value problems
\begin{equation}\label{sys:bdryab}
 \left\{
  \begin{aligned}
   \partial_{\alpha}\left(A_{ij}^{\alpha\beta}(x)\partial_{\beta}u^{j}+B_{ij}^{\alpha}u^{j}\right)
   +C_{ij}^{\beta}\partial_{\beta}u^{j}+D_{ij}u^{j}&=0&\mbox{in}&~\Omega_{1},& \\
   u&=0&\mbox{on}&~\Gamma_{1}^{+}\cup\Gamma_{1}^{-}.&
  \end{aligned}
 \right.
\end{equation}
We use the usual summation convention: $\alpha$ and $\beta$ are summed
from $1$ to $n$, while $i$ and $j$ are summed from $1$ to $N$. For $0<\lambda<\Lambda<\infty$, we assume that the
coefficients $A^{\alpha\beta}_{ij}(x)$ are measurable and bounded,
\begin{align}\label{coeff2}
|A^{\alpha\beta}_{ij}|\leq\Lambda,
\end{align}
and satisfy the rather weak
ellipticity condition
\begin{align}\label{coeff1}
\int_{\Omega_{1}}A^{\alpha\beta}_{ij}\partial_\alpha
\psi^i\partial_\beta\psi^j\geq\lambda \int_{\Omega_{1}}|\nabla\mathbf{\psi}|^2,
\quad\forall~\mathbf{\psi}\in H^1_0(\Omega_{1},\mathbb{R}^{N}).
\end{align}
Furthermore, we assume that $A_{ij}^{\alpha\beta}$, $B_{ij}^{\alpha}$, $C_{ij}^{\beta}$, $D_{ij}$, $h_{1}$ and $h_{2}$ are in $C^{k}(\Omega_{1})$ for some $k\geq0$, denote
$$\|A\|_{C^{k}(\Omega_{1})}+\|B\|_{C^{k}(\Omega_{1})}
+\|C\|_{C^{k}(\Omega_{1})}+\|D\|_{C^{k}(\Omega_{1})}\leq\beta_{k},$$
and
$$\|h_{1}\|_{C^{k}(\Omega_{1})}+\|h_{2}\|_{C^{k}(\Omega_{1})}\leq\gamma_{k},$$
where $\beta_{k}$ and $\gamma_{k}$ are some positive constants.
Hypotheses \eqref{coeff2} and \eqref{coeff1} are satisfied by
linear systems of elasticity (see
\cite{OSY}).

We give local estimates of
weak solutions $u$ of \eqref{sys:bdryab}, that is, $u\in{H}^{1}(\Omega_{1},\mathbb{R}^{N})$, $u=0$ on $\Gamma_{1}^{+}\cup\Gamma_{1}^{-}$ a.e., and
satisfies
$$\int_{\Omega_{1}}\left(A_{ij}^{\alpha\beta}(x)\partial_{\beta}u^{j}+B_{ij}^{\alpha}u^{j}\right)\partial_{\alpha}\zeta^{i}
-C_{ij}^{\beta}\partial_{\beta}u^{j}\zeta^{i}-D_{ij}u^{j}\zeta^{i}=0$$
for every vector-valued function
$\mathbf{\zeta}=(\zeta^{1},\cdots,\zeta^{N})\in{C}_{c}^{\infty}(\Omega_{1},\mathbb{R}^{N})$, and
hence for every $\mathbf{\zeta}\in{H}_{0}^{1}(\Omega_{1},\mathbb{R}^{N})$.

\begin{thm}\label{Thm1}
Assume the above, let $u\in{H}^1(\Omega_{1},\mathbb{R}^{N})$ be
a weak solution of \eqref{sys:bdryab}. Then for $k\geq0$, there exist constants $0<\mu<1$ and $C$, depending
only on $n,N,\lambda,\Lambda,k,\beta_{k+1+\frac{n}{2}}$ and $\gamma_{k+1+\frac{n}{2}}$, such that
$$|\nabla^{k}u(x)|\leq\,C\mu^{\frac{1}{\epsilon+|x'|}}\|u\|_{L^{2}(\Omega_{1})},
\quad\mbox{for~all}~~x=(x',x_{n})\in\Omega_{\frac{1}{2}}.$$
In particular,
$$\max_{-\frac{\epsilon}{2}+h_{2}(0')<x_{n}<\frac{\epsilon}{2}+h_{1}(0')}\big|\nabla^{k}u(0',x_{n})\big|
\rightarrow0,\quad\quad\mbox{as}~\epsilon\rightarrow0.$$
\end{thm}

A consequence of Theorem \ref{Thm1} is an extension of Theorem \ref{thmak} to all dimensions and to any smooth domains.

\begin{cor}\label{cor2}
Let $D_{1}$ and $D_{2}$ be two disjoint bounded open sets in $\mathbb{R}^{n}$, $n\geq2$, with $C^{k}$ boundaries for $k=\left[\frac{n}{2}+2\right]$, and $\mathrm{dist}(\partial{D}_{1},\partial{D}_{2})=\epsilon\in(0,1)$. Let $H$ be a
harmonic function in $\mathbb{R}^{n}\setminus(D_{1}\cup{D}_{2})$. Assume that $u$
satisfies
\begin{equation}\label{equ_ak2}
\begin{cases}
\Delta{u}=0&\mbox{in}~\mathbb{R}^{n}\setminus\overline{D_{1}\cup{D}_{2}},\\
u=0&\mbox{on}~\partial{D}_{1}\cup\partial{D}_{2},\\
\liminf\limits_{|x|\rightarrow\infty}|u(x)-H(x)|\leq\,K,&\mbox{for~some}~K>0.
\end{cases}
\end{equation}
Then there exists a constant $C$, depending only on $K$, $\|H\|_{L^{\infty}(\partial{D}_{1}\cup\partial{D}_{2})}$ and the $C^{k}$ norms and diameters of $D_{1}$ and $D_{2}$ (but independent of $\epsilon$), such that
\begin{equation}\label{adkl_u=02}
\|\nabla(u-H)\|_{L^{\infty}(\mathbb{R}^{n}\setminus\overline{D_{1}\cup{D}_{2}})}\leq\,C.
\end{equation}
\end{cor}

\section{Proof of Theorem \ref{Thm1}}\label{sys00}

In this section, we derive the $C^{k}$ estimates for
solutions of elliptic systems \eqref{sys:bdryab}. In the following, we first
show that the energy in $\Omega_{r}$ decays exponentially as $r$ tends to $0$. Unless otherwise stated, we use $C$ to denote some positive constants, whose values may vary from line to line, which depend only on $n,
N,\lambda,\Lambda$, $\beta_{0}$ and
$\gamma_{2}$, but is
independent of $\epsilon$.

\begin{lem} \label{lm:energy decay}
Let $u\in{H}^1(\Omega_{1},\mathbb{R}^{N})$ be a weak solution of
\eqref{sys:bdryab}, then there exist $0<\mu_0<1$ and $C$,
depending only on $n, N,\lambda,\Lambda$, $\beta_{0}$ and
$\gamma_{2}$, such that, for any $\sqrt{\epsilon}\leq{r}<\frac{1}{2}$,
\begin{equation}\label{energydecay}
\int_{\Omega_{r}}|\nabla{u}|^2dx\leq\,C(\mu_{0})^{\frac{1}{r}}\int_{\Omega_{1}}|\nabla{u}|^{2}.
\end{equation}
\end{lem}

\begin{proof} Without loss of generality, we can assume that
$\int_{\Omega_{1}}|\nabla{u}|^{2}=1$.
For any $0<t<s\leq 1$, we introduce a cutoff function
$\eta\in{C}^{\infty}(\Omega_{1})$ satisfying $0\leq\eta\leq\,1$, $\eta=1$ in
$\Omega_{t}$, $\eta=0$ in $\Omega_{1}\backslash\Omega_{s}$, and
$|\nabla\eta|\leq\frac{2}{s-t}$. Multiplying $(u\eta^2)$ on both
sides of the equation in $(\ref{sys:bdryab})$ and integrating by parts, we have
$$\int_{\Omega_{1}}\left(A_{ij}^{\alpha\beta}(x)\partial_{\beta}u^{j}+B_{ij}^{\alpha}u^{j}\right)\partial_{\alpha}(u^{i}\eta^2)
-C_{ij}^{\beta}\partial_{\beta}u^{j}(u^{i}\eta^2)-D_{ij}u^{j}(u^{i}\eta^2)=0.$$
Since
\begin{align*}
&\int_{\Omega_{s}}\left(A_{ij}^{\alpha\beta}(x)\partial_{\beta}u^{j}+B_{ij}^{\alpha}u^{j}\right)\partial_{\alpha}(u^{i}\eta^2)\\
&=\int_{\Omega_{s}}A_{ij}^{\alpha\beta}(x)\partial_{\beta}(\eta\,u^{j})\partial_{\alpha}(\eta\,u^{i})-
\int_{\Omega_{s}}A_{ij}^{\alpha\beta}(x)(u^{j}\partial_{\beta}\eta)\partial_{\alpha}(\eta\,u^{i})
+\int_{\Omega_{s}}B_{ij}^{\alpha}\eta\,u^{j}\partial_{\alpha}(u^{i}\eta)\\
&\hspace{5mm}+\int_{\Omega_{s}}A_{ij}^{\alpha\beta}(x)\partial_{\beta}(\eta\,u^{j})u^{i}\partial_{\alpha}\eta-
\int_{\Omega_{s}}A_{ij}^{\alpha\beta}(x)(u^{j}\partial_{\beta}\eta)(u^{i}\partial_{\alpha}\eta)
+\int_{\Omega_{s}}B_{ij}^{\alpha}(\eta\,u^{j})(\partial_{\alpha}\eta\,u^{i}),
\end{align*}
it follows, in view of \eqref{coeff1}, that
\begin{align*}
&\lambda\int_{\Omega_{s}}|\nabla(u\eta)|^{2}dx\\
&\leq\int_{\Omega_{s}}A^{\alpha\beta}_{ij}\partial_{\beta}(u^{j}\eta)\partial_{\alpha}(u^{i}\eta)dx\\
&=\int_{\Omega_{s}}A_{ij}^{\alpha\beta}(x)(u^{j}\partial_{\beta}\eta)\partial_{\alpha}(\eta\,u^{i})
-\int_{\Omega_{s}}B_{ij}^{\alpha}\eta\,u^{j}\partial_{\alpha}(u^{i}\eta)
-\int_{\Omega_{s}}A_{ij}^{\alpha\beta}(x)\partial_{\beta}(\eta\,u^{j})u^{i}\partial_{\alpha}\eta\\
&\hspace{5mm}+\int_{\Omega_{s}}A_{ij}^{\alpha\beta}(x)(u^{j}\partial_{\beta}\eta)(u^{i}\partial_{\alpha}\eta)
-\int_{\Omega_{s}}B_{ij}^{\alpha}(\eta\,u^{j})(\partial_{\alpha}\eta\,u^{i})\\
&\hspace{5mm}+\int_{\Omega_{s}}C_{ij}^{\beta}\partial_{\beta}(u^{j}\eta)(u^{i}\eta)
-\int_{\Omega_{s}}C_{ij}^{\beta}(u^{j}\partial_{\beta}\eta)(u^{i}\eta)+\int_{\Omega_{s}}D_{ij}u^{j}(u^{i}\eta^2)\\
&\leq\frac{\lambda}{4}\int_{\Omega_{s}}|\nabla(u\eta)|^{2}dx+
C\int_{\Omega_{s}}|u\nabla\eta|^{2}dx+C\int_{\Omega_{s}}|u\eta|^{2}dx.
\end{align*}
Since $u\eta=0$ on
$\Gamma_{1}^{-}$, by H\"{o}lder inequality, it follows that
\begin{align*}
\int_{\Omega_{s}}|u\eta|^{2}dx&=\int_{\Omega_{s}}\left(\int^{x_{n}}_{-\frac{\epsilon}{2}+h_{2}(x')}\partial_{n}(u\eta)(x',x^{n})dx_{n}\right)^{2}dx\nonumber\\
&\leq\int_{\Omega_{s}}\left((x_{n}+\frac{\epsilon}{2}-h_{2}(x'))\int^{x_{n}}_{-\frac{\epsilon}{2}+h_{2}(x')}|\partial_{n}(u\eta)|^{2}dx_{n}\right)dx\nonumber\\
&\leq\,C(\epsilon+s^{2})^{2}\int_{\Omega_{s}}|\nabla(u\eta)|^{2}dx.
\end{align*}
Taking $0<\delta_{0}<1$ such that $C^{2}(\delta_{0}+\delta_{0}^{2})^{2}=\frac{\lambda}{4}$, then we have
\begin{equation}\label{lem21.1}
\int_{\Omega_{s}}|\nabla{u}|^{2}\eta^{2}dx\leq\,C\int_{\Omega_{s}}u^{2}|\nabla\eta|^{2}dx,\quad\mbox{for}~~\epsilon,s<\delta_{0}.
\end{equation}
Again using $u=0$ on
$\Gamma_{1}^{-}$, and by H\"{o}lder inequality, we have
\begin{align}\label{lem21.2}
\int_{\Omega_{s}}u^{2}dx\leq\,C(\epsilon+s^{2})^{2}\int_{\Omega_{s}}|\nabla{u}|^{2}dx.
\end{align}
Combining \eqref{lem21.1} and \eqref{lem21.2}, we have
\begin{align}\label{iterateformula}
\int_{\Omega_{t}}|\nabla{u}|^{2}dx\leq\,C\left(\frac{\epsilon+s^{2}}{s-t}\right)^{2}\int_{\Omega_{s}}|\nabla{u}|^{2}dx,
\quad\mbox{for}~~s<\delta_{0}.
\end{align}
For simplicity of notation, we denote
$$F(t)=\int_{\Omega_{t}}|\nabla{u}|^{2}dx,$$
then \eqref{iterateformula} can be written as
\begin{equation}\label{lem21.3}
F(t)\leq C\left(\frac{\epsilon+s^{2}}{s-t}\right)^{2}F(s).
\end{equation}
For $\sqrt{\epsilon}\leq\,t<s\leq\delta_{0}$, we have the following
iterative formula
$$
F(t)\leq \left(\frac{C_{0}s^{2}}{s-t}\right)^{2}F(s),
$$
where $C_{0}$ is a fixed constant, depending only on
$n,N,\lambda,\Lambda$, $\beta_{0}$ and $\gamma_{2}$. Let
$\delta=\min\{\frac{1}{8C_{0}},\delta_{0}\}$ and $t_0=r<\delta$,
$t_{i+1}=2\delta(1-\sqrt{1-t_{i}/\delta})$ if $t_i \leq
\delta$. Then
\begin{equation}\label{lem21.t}
\frac{C_{0}t_{i+1}^{2}}{t_{i+1}-t_{i}}=\frac{1}{2},
\end{equation}
and $\{t_{i}\}$ is an increasing sequence. It is easy to see that for some $i$, $t_{i}>\delta$.
Let $k$ be the integer satisfying $t_k\leq \delta$ and
$t_{k+1}>\delta$. Clearly $t_{k+1}\leq2\delta$. Then for any $0\leq i\leq k$, we have
\begin{align}\label{lem21.6}
F(t_{i})\leq\left(\frac{C_{0}t_{i+1}^2}{t_{i+1}-t_{i}}\right)^{2}F(t_{i+1})=\frac{1}{4}F(t_{i+1}),
\end{align}
Iterating \eqref{lem21.6} $k$-times, we have
\begin{equation}\label{lem21.4}
F(t_{0})\leq(\frac{1}{4})^{k+1}F(t_{k+1})\leq(\frac{1}{4})^{k+1}F(2\delta)\leq(\frac{1}{4})^{k+1}.
\end{equation}

Now we estimate $k$. From \eqref{lem21.t} it follows that
$$\frac{1}{2C_{0}t_i}=\frac{1}{2C_{0}t_{i+1}}+\frac{1}{1-2C_{0}t_{i+1}},\quad\mbox{for}~~0\leq\,i\leq\,k.$$
Then summing it from $i=0$ to $i=k$, we have
$$\frac{1}{2C_{0}t_0}=\frac{1}{2C_{0}t_{k+1}}+\sum_{i=1}^{k+1}\frac{1}{1-2C_{0}t_i}.$$
Since $0<t_{i}\leq\delta\leq\frac{1}{8C_{0}}$ for $1\leq i\leq k$, it follows that
$$1<\frac{1}{1-2C_{0}t_{i}}\leq\frac{4}{3}.$$
Then
$$k+1<\frac{1}{2C_{0}}\left(\frac{1}{t_0}-\frac{1}{t_{k+1}}\right)<\frac{4}{3}(k+1).$$
Recalling $t_0=r$, and $\delta<t_{k+1}\leq2\delta$, we have
$$\frac{3}{8C_{0}r}-3\leq\,k+1<\frac{1}{2C_{0}r}-2.$$
Therefore, from \eqref{lem21.4},
\begin{equation*}
F(r)=F(t_{0})\leq\left(\frac{1}{4}\right)^{k+1}\leq
\left(\frac{1}{4}\right)^{\frac{3\delta}{r}-3}.
\end{equation*}
The proof is completed.
\end{proof}

\begin{proof}[Proof of Theorem \ref{Thm1}]

Given a point $z=(z',z_{n})\in\Omega_{1}$, define
\begin{equation}\label{hatomega}
\widehat{\Omega}_{s}(z):=\left\{x=(x',x_{n})\in\Omega_{1}\big|~-\frac{\epsilon}{2}+h_{2}(x')<x_{n}<\frac{\epsilon}{2}+h_{1}(x'),~|x'-z'|<s\right\}.
\end{equation}
We consider the following scaling in
$\widehat{\Omega}_{\frac{1}{2}(\epsilon+h_{1}(z')-h_{2}(z'))}(z)$,
\begin{equation}
\left\{
\begin{aligned}
&Ry'+z'=x', \\
&Ry_{n}-\frac{\epsilon}{2}+h_{2}(z')=x_n,
\end{aligned}
\right. \nonumber
\end{equation}
where $R=(\epsilon+h_{1}(z')-h_{2}(z'))$. Denote
$$\widehat{h}_{1}(y'):=\frac{1}{R}
\left(\epsilon-h_{2}(z')+h_{1}\left(z'+Ry'\right)\right),$$
$$
\widehat{h}_{2}(y'):=\frac{1}{R}
\left(-h_{2}(z')+h_{2}\left(z'+Ry'\right)\right)$$
Then
$$\widehat{h}_{1}(0')=1,\quad\widehat{h}_{2}(0')=0,$$
and
$$\widehat{h}_{2}(y')<\widehat{h}_{1}(y'),\quad|\nabla^{l}\widehat{h}_{1}(y')|,|\nabla^{l}\widehat{h}_{2}(y')|\leq\,C_{l},
~~\mbox{for~}|y'|\leq1,~l\geq1.$$
Let
$$\widehat{u}(y',y_{n})=u\left(Ry'+z',
Ry_{n}-\frac{\epsilon}{2}+h_{2}(z')\right),$$ then
$\widehat{u}(y)$ satisfies
\begin{equation}\label{tildeu}
\partial_{\alpha}\left(\widehat{A}_{ij}^{\alpha\beta}(y)\partial_{\beta}\widehat{u}^{j}(y)+\widehat{B}_{ij}^{\alpha}(y)\widehat{u}^{j}(y)\right)
   +\widehat{C}_{ij}^{\beta}(y)\partial_{\beta}\widehat{u}^{j}(y)+\widehat{D}_{ij}(y)\widehat{u}^{j}(y)=0
\quad\mbox{in}~\,Q_{1},
\end{equation} where
$$\widehat{A}(y)=A\left(Ry'+z', Ry_{n}-\frac{\epsilon}{2}+h_{2}(z')\right),\quad\widehat{B}(y)=RB\left(Ry'+z', Ry_{n}-\frac{\epsilon}{2}+h_{2}(z')\right),$$
$$\widehat{C}(y)=RC\left(Ry'+z', Ry_{n}-\frac{\epsilon}{2}+h_{2}(z')\right),\quad \widehat{D}(y)=R^{2}D\left(Ry'+z', Ry_{n}-\frac{\epsilon}{2}+h_{2}(z')\right),$$
and for $r<1$,
\begin{align*}
Q_{r}:=\bigg\{~(y',y_{n})\in\mathbb{R}^{n}~\big|~~\widehat{h}_{2}(y')<y_{n}<\widehat{h}_{1}(y'),~|y'|<r\bigg\}.
\end{align*}

Using $L^2$ estimates for elliptic systems \eqref{tildeu} and by the Sobolev imbedding theorems, we
have
\begin{align*}
\max_{0\leq{y}_{n}\leq1}|\nabla^{k}\widehat{u}(0',y_{n})|\leq\,C\|\nabla\widehat{u}\|_{L^{2}(Q_{1})},
\end{align*}
where $C$ depends only on $n,N,\lambda$, $\Lambda,k$, $\beta_{k+\frac{n}{2}+1}$ and
$\gamma_{k+\frac{n}{2}+1}$. It follows, in view of Lemma \ref{lm:energy decay}, that
\begin{align}\label{lem21.5}
|\nabla^{k}u(z)|\leq\,C\left(\epsilon+|z'|^2\right)^{1-k-\frac{n}{2}}\|\nabla{u}\|_{L^{2}\left(\Omega_{|z'|+\frac{R}{2}}\right)}
\leq\,C\left(\epsilon+|z'|^2\right)^{1-k-\frac{n}{2}}(\mu_0)^{\frac{1}{|z'|+\frac{R}{2}}},
\end{align}
where $\mu_{0}<1$ was defined in Lemma \ref{lm:energy decay}, and $C$ depends only on $n,N,\lambda$, $\Lambda,k$, $\beta_{k+\frac{n}{2}+1}$ and
$\gamma_{k+\frac{n}{2}+1}$. The proof is completed.
\end{proof}

\begin{proof}[Proof of Corollary \ref{cor2}] Without loss of generality, we assume that $D_{1}$ and $D_{2}$ are separated by the plane $x_{n}=0$, with $(0',\frac{\epsilon}{2})\in\partial{D}_{1}$ and $(0',-\frac{\epsilon}{2})\in\partial{D}_{2}$.
Since
$$\Delta(u-H)=0~~\mbox{in}~\mathbb{R}^{n}\setminus\overline{D_{1}\cup{D}_{2}}.$$
We have, after applying the maximum principle to $u-H$, that
$$|u-H|\leq\|H\|_{L^{\infty}(\partial{D}_{1}\cup\partial{D}_{2})}+K,\quad\mbox{in}~\mathbb{R}^{n}\setminus\overline{D_{1}\cup{D}_{2}}.$$
Corollary \ref{cor2} then follows from Theorem \ref{Thm1}.
\end{proof}

\noindent{\bf{\large Acknowledgements.}}  The first author was
partially supported by NSFC (11071020) (11126038), SRFDPHE
(20100003120005) and Ky and Yu-Fen Fan Fund Travel Grant from the
AMS. He also would like to thank the Department of Mathematics and
the Center for Nonlinear Analysis at Rutgers University for the
hospitality and the stimulating environment. The work of the second
author was partially supported by NSF grant DMS-0701545, DMS-1065971 and DMS-1203961. The first
and second author were both partially supported by Program for
Changjiang Scholars and Innovative Research Team in University in
China (IRT0908).

\end{document}